\documentclass[12pt]{article}

\usepackage{amsfonts, amssymb, amsmath, amsthm}
\usepackage{mathrsfs}

\usepackage{fancyhdr}

\usepackage{tikz}
\usepackage[all]{xy}

\usetikzlibrary{arrows}
\usepackage{hyperref}

\usepackage{color,hyperref}
    \catcode`\_=11\relax
    
    \def\_email#1@#2\q_nil{%
      \href{mailto:#1@#2}{{\emailfont #1\emailampersat #2}}
    }
    \newcommand\emailfont{\sffamily}
    \newcommand\emailampersat{{\color{red}\small@}}
    \catcode`\_=8\relax  
    
        \theoremstyle{plain}
\newtheorem{thm}{Theorem}[section]
\newtheorem{lemma}[thm]{Lemma}

\newtheorem{example}[thm]{Example}

\theoremstyle{definition}

 \newcommand{\be}{\begin{equation}} 
\newcommand{\ee}{\end{equation}}

\newcommand{\one}{\mathbf{1}}
\newcommand{\Set}{\mathbf{Set}}

\newcommand{\Meas}{\mathbf{Meas}}

\newcommand{\T}{\mathcal{P}}
\newcommand{\Cvx}{\mathbf{Cvx}}
\newcommand{\SCvx}{\mathbf{SCvx}}

\newcommand{\G}{\mathcal{G}}

\newcommand{\C}{\mathcal{C}}

\newcommand{\Rinf}{\mathbb{R}_{\infty}}
\newcommand{\Yop}{{\mathcal{Y}^{op}}|}

    \title{\vspace{-2.85cm}Erratum and addendum:  The factorization of the Giry monad}     

\author{Kirk Sturtz}
\date{}

\fancypagestyle{plain}{
\fancyhf{} 
\fancyfoot[C]{\thepage} 
  
}

\begin{document}
\thispagestyle{empty} 

\maketitle
\begin{abstract} The category of super convex spaces,  a proper subcategory of convex spaces, possesses the property that it has a codense subcategory.   This codense subcategory allows for  an elementary proof that the Giry monad factorizes through the category of super convex spaces, not the category of convex spaces as erroneously claimed in an earlier article.  
\end{abstract}

\section{Introduction}  \label{s1}
  
   In this article, we correct an error in our article \emph{A factorization of the Giry monad}, thereby showing that the Giry monad $\G$ factors through the category of super convex spaces, $\SCvx$, not the category of convex spaces, $\Cvx$.  We show that by using a codense subcategory of $\SCvx$,   the barycenter map, which is the counit of our adjunct factorization of the Giry monad, is easily constructed, and avoids  requiring any assumptions concerning the existence or non existence of measurable cardinals.   
   
   In this introduction, we give the definition of a super convex space, and in \S \ref{s2}  show how the factorization arises from the existence of a codense subcategory of $\SCvx$.  This article is self-contained, and the few proofs that we make use of are relegated to \S \ref{s3}. 
  At present, we do not know whether $\SCvx$ is equivalent to the category of $\G$-algebras or not.  We hope others will be stimulated by these results and investigate this open question.   
  
  \vspace{.1in}
   
As the category $\SCvx$ is not well know, let us give the definition, as per  B\"orger and Kemper\cite{Borger}.  Let $\Omega$ denote the set of all countable partitions of one,
 \be \nonumber
 \Omega = \{ \alpha = \{\alpha_i\}_{i=1}^{\infty} \, | \,    \sum_{i=1}^{\infty} \alpha_i=1, \alpha_i \in [0,1] \},
 \ee 
where $\sum_{i=1}^{\infty} \alpha_i = $ is shorthand notation for the limit condition, $\displaystyle{ \lim_{N \rightarrow \infty}} \{\sum_{i=1}^N \alpha_i\} =1$.   
An $\Omega$-algebra is a set $A$ together with a map $\Omega \rightarrow \Set(A^{\mathbb{N}}, A)$,  $\alpha \mapsto \alpha_A$.  A morphism from an $\Omega$-algebra $A$ to an $\Omega$-algebra $B$ is a set map $A \stackrel{m}{\longrightarrow} B$ satisfying $m \circ \alpha_A = \alpha_B \circ m^{\mathbb{N}}$ for all $\alpha \in \Omega$, where $A^{\mathbb{N}} \stackrel{m^{\mathbb{N}}}{\longrightarrow} B^{\mathbb{N}}$ is defined componentwise.  The $\Omega$-algebras form a category, with composition of morphisms being the set-theoretical one.

With the convention $\alpha_A(a) \stackrel{def}{=} \sum_{i \in \mathbb{N}} \alpha_i a_i$, where $a=\{a_i\}_{i \in \mathbb{N}}$, and $\alpha \in \Omega$, an $\Omega$-algebra $A$ is a super convex space provided the following two axioms are satisfied:
\begin{enumerate}
\item $\sum_{i \in \mathbb{N}} \delta_i^j a_i = a_j$ for all $j \in \mathbb{N}$, and all $\{a_i\}_{i \in \mathbb{N}} \in A^{\mathbb{N}}$.
\item $\sum_{i \in \mathbb{N}} \alpha_i ( \sum_{j \in \mathbb{N}} \beta_j^i a_j) = \sum_{j \in \mathbb{N}}( \sum_{i \in \mathbb{N}} \alpha_i \beta_j^i)a_j$ for all $\alpha, \beta^i \in \Omega$.
\end{enumerate}

Informally,   the category $\SCvx$  has as objects, a set $A$ with a structural map \mbox{$\Omega \rightarrow \Set(A^{\mathbb{N}}, A)$}, which specifies how any sequence of elements in $A$, and element $\{\alpha_i\}_{i=1}^{\infty} \in \Omega$,  combine to yield an element in that set $A$.  The first condition above says that any countable sum of the form $0a_1 + 0a_2 + \ldots 0 a_{j-1} + 1 a_j + 0 a_{j+1} + \ldots$   must be equal to $a_j$.  The second condition is the associativity condition.  The  arrows of $\SCvx$ preserve countable convex sums, so $A \stackrel{m}{\longrightarrow} B$ is a morphism in $\SCvx$ if and only if
\be \nonumber
m(\sum_{i=1}^{\infty} \alpha_i a_i ) = \sum_{i=1}^{\infty} \alpha_i m(a_i)  \quad \quad \textrm{where } \, \sum_{i=1}^{\infty} \alpha_i = 1, \, \alpha_i \in [0,1],
\ee
where the elements $a_i$ all lie in $A$.   

Obviously, $\SCvx$ is a full subcategory of $\Cvx$.

\section{Using codensity to factorize the Giry monad} \label{s2}
 To say a subcategory $\C$ of $\SCvx$ is codense is equivalent to saying that the  truncated Yoneda mapping $\SCvx^{op} \stackrel{\Yop}{\longrightarrow} \Set^{\C}$
 is (still) full and faithful.   Isbell\cite{Isbell}, using the truncated Yoneda mapping, used the terminology of ``right adequate''  rather than a codense subcategory.  While we use the standard terminology of ``codense'', we prefer to think of this concept as defined in terms of  the truncated Yoneda mapping,  as opposed to defining it in terms of a limit  of a functor, e.g., see MacLane\cite[p242]{Mac}.  This perspective is  useful because our problem, as well as many others,  reduces to the question of whether a ``generalized point'' of an object $A$, meaning a  natural transformation $P \in \mathbf{Nat}(Hom(A, \cdot), Hom(\one, \cdot))$, corresponds to a point $a \in A$, i.e., whether $\Yop(1 \stackrel{a}{\longrightarrow} A)=P$.
 
   In $\SCvx$, the full subcategory $\C$ of $\SCvx$, consisting of the single object $\Rinf$, the one point extension of the real line, 
is a codense category of $\SCvx$.\footnote{The super convex structure of $\Rinf$ is an extension of the convex structure of $\Rinf$, which is defined,  for all $u \in (-\infty,\infty)$ and all $r \in (0,1]$,  by  $(1-r) u + r \infty = \infty$. The extension is defined, for all $u_i \in \Rinf$,  by 
 \be \nonumber
 \sum_{i=1}^{\infty} \alpha_i u_i = \left\{  \begin{array}{ll}\displaystyle{ \lim_{N \rightarrow \infty}} \{\sum_{i=1}^N \alpha_i u_i\} & \textrm{provided the limit exist} \\ \infty & \textrm{otherwise} \end{array} \right..
 \ee
 }    
 Since $\C$ consist of a single object, we abuse terminology and just say $\Rinf$ is codense in $\SCvx$.   Because $\Rinf$ is codense it follows that any $\SCvx$ map $\Rinf^A \stackrel{P}{\longrightarrow} \Rinf$ is an evaluation point, $\Yop(a) = P$ for a unique element $a \in A$.  Because the object $\Rinf$ is a coseparator in $\Cvx$, and hence also a coseparator in $\SCvx$, the faithful property of $\Yop$ is immediate.\footnote{ The fact that $\mathbb{R}_{\infty}$ is a cogenerator is due to B\"orger and Kemper\cite{Borger}.}  

 We now illustrate how this ``generalized element''  condition arises in the factorization of the Giry monad,  $(\G, \eta, \mu)$.  For any measurable space $X$, the set  $\G(X)$, consisting of all probability measures on $X$, has a natural super convex space structure associated with it,  since given any sequence $\{P_i\}_{i=1}^{\infty}$ of probability measures on $X$, and any $\alpha \in \Omega$, $\sum_{i=1}^{\infty} \alpha_i P_i \in \G(X)$.
 Consequently, we can view the functor $\G$ as a functor $\T$ into the category $\SCvx$, and define a functor $\Sigma$

   \begin{equation}   \nonumber
 \begin{tikzpicture}[baseline=(current bounding box.center)]

	\node	(M)	at	(0,0)	               {$\Meas$};
	\node	(C)	at	(4,0)        {$\SCvx$};
	\node         (c)    at       (6,0)        {$\T \dashv \Sigma$};
           \draw[->,above] ([yshift=2pt] M.east) to node {$\T$} ([yshift=2pt] C.west);
           \draw[->,below] ([yshift=-2pt] C.west) to node {$\Sigma$} ([yshift=-2pt] M.east);

 \end{tikzpicture}
 \end{equation}
 \noindent
by assigning to each super convex space $A$, the measurable space $(|A|, \langle \SCvx(A, \Rinf) \rangle)$ where the $\sigma$-algebra, associated with the underlying set $|A|$, is the smallest \mbox{$\sigma$-algebra} generated by $\SCvx(A,\Rinf)$, with $\Rinf$ having the standard Borel $\sigma$-algebra defined on $\mathbb{R}$ extended by the measurable set $\{\infty\}$.

The unit of the adjunction is the natural transformation \mbox{$id_{\Meas} \stackrel{\eta}{\Rightarrow} \Sigma \circ \T$},  at component $X$, just sends an element to the Dirac measure at that point, $x \mapsto \delta_x$.  
To define the counit, $\T \circ \Sigma \stackrel{\epsilon}{\Rightarrow} id_{\SCvx}$, at  component $A$, observe that 
every probability measure $P \in \T(\Sigma A)$ defines an  operator $\hat{P}$ on the space of measurable functions $\Meas(\Sigma A, \Rinf)$, via $f \mapsto \int_A f \, dP$, and this operator preserves countable convex sums.
  By construction, $\SCvx(A, \Rinf) \subseteq \Meas(\Sigma A, \Rinf)$, and both are objects in $\SCvx$. Hence we can take the restriction of $\hat{P}$  to $\SCvx(A, \Rinf)$.    This restriction, $\hat{P}| = \hat{P} \circ \iota$,

\begin{equation}   \label{diagram}
 \begin{tikzpicture}[baseline=(current bounding box.center)]
 
 \node        (sub)  at  (0, -1.6)   {$\Rinf^A$};
  \node     (IA) at  (0,0)  {$\Rinf^{\Sigma A}$};
   \node   (I)  at   (3,0)   {$\Rinf$};
   \node    (d)  at   (1, -2.3) {in $\SCvx$};
   \node     (c)    at  (5.5, 0)   {\small{where}};
   \node    (c1)  at  (6.5, -.6)  {\small{$\Rinf^{\Sigma A}\stackrel{def}{=} \Meas(\Sigma A, \Rinf)$}};
   \node      (c2)  at  (6.3, -1.3)  {\small{$\Rinf^A \stackrel{def}{=} \SCvx(A, \Rinf)$ }};
   
     \draw[>->,left] (sub) to node {\tiny{$\iota$}} (IA);
  \draw[->,above] ([yshift=2pt] IA.east) to node [xshift=-5pt] {\tiny{$\hat{P} = \int_A \cdot \, dP$}} ([yshift=2pt] I.west);
  \draw[->,below, dashed] (sub) to node [xshift=23pt] {\tiny{$\hat{P}| = ev_a$}} (I);
 \end{tikzpicture}
 \end{equation}
\noindent
is an arrow in $\SCvx$, which lies in the image of the Yoneda mapping\footnote{Because both $\one$ and $A$ are  super convex spaces, and $\Yop$ is full.}, and hence because the object $\Rinf$ is codense in $\SCvx$,  it follows that $\Rinf^A \stackrel{\hat{P}|}{\longrightarrow} \Rinf$ is a point of $A$,  $\hat{P}|= \Yop(a) = ev_a$, for a  unique element $a \in A$.  Note that $\hat{P}$ does not, in general, lie in the image of the Yoneda mapping and hence need not be a point of $A$.
 
Thus we obtain a mapping $\T(\Sigma A) \stackrel{\epsilon_A}{\longrightarrow} A$ sending a probability measure $P$ to the unique element $a \in A$ for which $\hat{P}| = ev_a$.  
This construction is natural in the argument $A$ (Lemma \ref{naturalA}).  The pair of natural transformations, $\eta$ and $\epsilon$, satisfy the triangle equality conditions, and hence we can conclude $\T \dashv \Sigma$ (Theorem \ref{adjunction}).  The counit of that adjunction,  at each component $A$,  $\T(\Sigma A) \stackrel{\epsilon_A}{\longrightarrow} A$ is called the \emph{barycenter mapping} at component $A$.

For brevity, let $\Rinf^{\Rinf} \stackrel{def}{=} \SCvx(\Rinf, \Rinf)$, and let $\SCvx_{\tiny{\Rinf^{\Rinf}}}(\Rinf^A, \Rinf)$ denote the subset of $\SCvx(\Rinf^A, \Rinf)$ satisfying the  condition 
\be \nonumber
J(g \circ m) = g( J(m)) \quad  \forall m \in \Rinf^A, \, \, \forall g \in \Rinf^{\Rinf}.
\ee
The elements $J \in \SCvx_{\tiny{\Rinf^{\Rinf}}}(\Rinf^A, \Rinf)$ are referred to as $\Rinf$-generalized elements of $A$.\cite[Def. 8.19]{LR} An $\Rinf$-generalized element  of $A$ is precisely a natural transformation  $\Set^{\C}(\SCvx(A, \_), \SCvx(\one,\_))$, evaluated at the single component $\Rinf$, and the condition above is the naturality requirement.

With this terminology, the barycenter mapping arises from the fact that the restriction of every probability measure $P$ on a super convex space $A$ is a $\Rinf$-generalized point of $A$, lying in the image of $\Yop$.  By the codensity of $\Rinf$ in $\SCvx$, it follows that the $\Rinf$-generalized point of $A$ is a point.

The fullness condition of $\Yop$ follows from the fact that  a  $\Rinf$-generalized point of $A$, say $J$, satisfies the property that for any $m \in \SCvx(A, \Rinf)$, $J(m)$ lies in the image of the map $m$, $J(m) \in Image(m)$.  This  property does not hold, in general,  for $\Cvx$ as the following example shows.

\begin{example}
Let $P$ denote the half-Cauchy distribution on $\mathbb{R}_+ = [0,\infty)$.  The inclusion map $\mathbb{R}_+ \hookrightarrow \Rinf$ is a convex map, and, as is well known, the expectation of the half-Cauchy probability measure, given by $\int_{\mathbb{R}_+} \iota \, dP$ does not exist, and hence it is impossible to construct a barycenter map $\T(\Sigma \mathbb{R}_+) \rightarrow \mathbb{R}_+$.   This is due to the fact that $\mathbb{R}_+$ is not a super convex space. By choosing $\alpha \in \Omega$ by $\alpha_i = \frac{1}{2^i}$ and taking the points  $x_i = i 2^i \in \mathbb{R}_+$ it follows that the countable sum $\sum_{i=1}^{\infty} \frac{1}{2^i} i 2^i$ does not lie in $\mathbb{R}_+$. 

Note that the inability to construct a barycenter map for  $\mathbb{R}_+$ is not a question of compactness because the open unit interval, $(0,1)$, with the natural super convex space structure,  is not compact either, but it is a super convex space, and hence there exist a barycenter map for the space $(0,1)$.

\end{example}
The next section is devoted to proving, in order, the following properties, referred to above. 
 \begin{enumerate}
 \item For $A$ a super convex space, every $\Rinf$-generalized point $J$ of $A$ satisfies the property that $J(m) \in Image(m)$. 
 \item The object $\Rinf$ is codense in $\SCvx$. 
 \item For $X$ any measurable space, the two super convex spaces,  $\G(X)$ and $\SCvx_{\Rinf^{\Rinf}}(\Rinf^X, \Rinf)$, are $\SCvx$-isomorphic.
 \item The  naturality of the counit $\epsilon$ for $\T$ and $\Sigma$.
 \item  The adjunction, $\T \dashv \Sigma$. 
 \end{enumerate}
 
\section{Proofs} \label{s3}

\begin{lemma} \label{imageLemma}
If $A \in \SCvx$ then every  $\Rinf$-generalized element of $A$ satisfies the property that 
\be \nonumber
J(m) \in Image(m) \quad \forall m \in \Rinf^A.
\ee
\end{lemma}
\begin{proof}
For every $m \in \Rinf^A$, the image of $m$ is a super convex (sub)space of $\Rinf$.  Hence it is of one of the following forms.
\begin{enumerate}
\item The image of $m$ is an interval, $\{ (u,v), [u,v), [u,v], (u,v] \}$ where $u > -\infty$ and $v < \infty$.
\item The image of $m$ is an interval with the right endpoint $\infty$, $Image(m) \in \{ (u, \infty], [u, \infty] \}$.
\end{enumerate}
The function space $\Rinf^A$ has a partial order structure $<$ defined on it by $f < g$ if and only if $f(a) < g(a)$ for all $a \in A$.
Similiarly, there is a partial order structure $\le$ defined on $\Rinf^A$ by $f \le g$ if and only if $f(a) \le g(a)$ for all $a \in A$.
Any $\Rinf$-generalized element of $A$  preserves this ordering, where $\Rinf$ has the obvious partial orders structures $<$ and $\le$ also.  Now let $A \stackrel{\overline{u}}{\longrightarrow} \Rinf$ denote the constant map on $A$ with value $u$. Because $J$ is a $\Rinf$-generalized point of $A$ it follows that $J$ is a weakly averaging functional, $J(\overline{u})=u$.  

Suppose the image of $m$ is $[u, v)$, with $u> - \infty$ and $v< \infty$.  Then $\overline{u} \le m$ and $m < \overline{v}$, and because $J$ preserves the partial ordering structures, it follows that $u \le J(m) < v$.  Consequently $J(m) \in Image(m)$.
The other cases are handled similarly.
\end{proof}

\begin{thm} The full subcategory of $\SCvx$ consisting of the single object $\Rinf$ is a codense subcategory in $\SCvx$.
\end{thm}
\begin{proof}
The fact that $\Rinf$ is the coseparator for $\SCvx$ implies that $\Yop$ is faithful.
To prove it is full, let  $A$ be any convex space and let $J$ be a $\Rinf$-generalized element of $A$.   
By Lemma \ref{imageLemma}, for every countably convex  map $m \in \Rinf^A$, $J(m)$ lies in the image of $m$, and hence there exist at least one $a \in A$ such that
\be \nonumber
J(m) = ev_a(m).
\ee
We need to prove that there exist an $a \in A$ such that $J=ev_a$, so that the choice of $a$ does not depend upon $m$.
 
For every $m \in \Rinf^A$, define an equivalence relation $\sim_m$ on $A$ by $a_1 \sim_m a_2$ if and only if $m(a_1)=m(a_2)$.  It follows that if $a_1 \sim_m a_2$ then $m(a_1) = J(m)=m(a_2)$.


Now suppose $m_1, m_2 \in \Rinf^A$ are two elements  for which $J(m_1) = m_1(a_1)$ and $J(m_2) = m_2(a_2)$.  Consider the quantity
$J((1-r) m_1 + r m_2)$.  Since $J((1-r) m_1 + r m_2)$ lies in the image of $(1-r) m_1 + r m_2$, there exist some $a_3 \in A$ such that 
\be \nonumber
J((1-r)m_1 + r m_2) = ((1-r)m_1 + r m_2)(a_3) = (1-r) m_1(a_3) + r m_2(a_3).
\ee
  But since $J$ is a countably convex map, we also have the equality
  \be \nonumber
  J((1-r)m_1 + r m_2) = (1-r)J(m_1) + r J(m_2) = (1-r)m_1(a_1) + r m_2(a_2).
  \ee
Equating the last two expressions for $J((1-r)m_1 + r m_2)$, we obtain 
\be \nonumber
(1-r)m_1(a_1) + r m_2(a_2) = (1-r) m_1(a_3) + r m_2(a_3).
\ee
Since this equation holds for all $r \in [0,1]$, it follows that  
\be \nonumber
m_1(a_1)=J(m_1) = m_1(a_3)  \quad \textrm{ and } \quad  m_2(a_2) =J(m_2)= m_2(a_3).
\ee
In terms of the relations $\sim_{m_1}$ and $\sim_{m_2}$, the above equations can be expressed as
\be \nonumber
a_1 \sim_{m_1} a_3  \quad \textrm{and} \quad a_3 \sim_{m_2} a_2
\ee
implying that for any two elements in $m_1, m_2 \in \Rinf^A$ there exist an element in $a_3 \in A$ satisfying $J(m_1) = ev_{a_3}(m_1)$ and $J(m_2) = ev_{a_3}(m_2)$.  Consequently, we conclude  there is a unique element $a \in A$ such that $J = ev_a$.


\end{proof}

\begin{lemma} \label{naturalA} The mapping $\T(\Sigma A) \stackrel{\epsilon_A}{\longrightarrow} A$ 
 defined by $P \mapsto a$, where the restriction of the operator $\hat{P}$  satisfies $\hat{P}| = ev_a$,
specifies the components of a natural transformation.
\end{lemma}
\begin{proof}
Let $m \in \SCvx(A, B)$. 
To prove naturality we must show that \mbox{$m(\epsilon_A(P|)) = \epsilon_B((Pm^{-1})|)$}, where $(Pm^{-1})|$ is the pushforward probability measure on $\Sigma B$, restricted to operate on the set
\be \nonumber
\SCvx(B, \Rinf) \subseteq \Meas(\Sigma B, \Rinf).
\ee
  Since $\epsilon_B(Pm^{-1}|)$ is the unique element in $B$ satisfying the equation
  \be \nonumber
  (P|m^{-1})(k) = \epsilon_B(P|m^{-1})(k), \quad \textrm{ for all }k \in \SCvx(B, \Rinf),
  \ee
 we have
\mbox{$P|(k \circ m) = k(\epsilon_B(P|m^{-1}))$}.  Since $k \circ m \in \SCvx(A, \Rinf)$ it follows by the uniqueness of $\epsilon_A(P|)$, that
\be \nonumber
(k \circ m)(\epsilon_A(P|)) = k( \epsilon_B(Pm^{-1}|) ) \quad \forall k \in \SCvx(B, \Rinf).
\ee
It therefore follows that $m(\epsilon_A(P|)) = \epsilon_B((Pm^{-1})|)$, thereby proving naturality.
 
\end{proof}

\begin{lemma} \label{lemmaIso}
The map $\G(X) \stackrel{\phi}{\longrightarrow} \SCvx_{\Rinf^{\Rinf}}(\Rinf^X, \Rinf)$,  specified by $\phi(P)(\chi_{U}) = P(U)$, makes the diagram
\begin{equation}   \label{diagram}
 \begin{tikzpicture}[baseline=(current bounding box.center)]
          \node   (GX)   at   (0,0)    {$\G(X)$};
          \node    (SC) at     (0, -2)  {$\SCvx_{\Rinf^{\Rinf}}(\Rinf^X, \Rinf)$};
          \node    (R)    at     (3, -1)    {$\Rinf$};
          \node     (c)   at     (6.7, -1)   {$\phi(P)(\chi_{U}) = P(U) \quad \forall P \in \G(X)$};
          \draw[->,above,right] (GX) to node [yshift=3pt]{$ev_U$} (R);
          \draw[->,below, right] (SC) to node [yshift=-5pt, xshift=-2pt] {$ev_{\chi_U}$} (R);
          \draw[->,left] ([xshift=-2pt] GX.south) to node {$\phi$} ([xshift=-2pt] SC.north);
          \draw[->,right,dashed] ([xshift=2pt] SC.north) to node {$\phi^{-1}$} ([xshift=2pt] GX.south);
 \end{tikzpicture}
 \end{equation}
 commute, and $\phi$
 is an isomorphism of super convex spaces.
\end{lemma}
\begin{proof}
The commutativity of the diagram, without  the inverse mapping,  is the definition of $\phi$.

  The inverse of this map sends an element $\Rinf^X \stackrel{J}{\longrightarrow} \Rinf$ to the probability measure $\Sigma_X \stackrel{\phi^{-1}(J)}{\longrightarrow} [0,1]$ defined by $(\phi^{-1}{J})U = J(\chi_U)$.  This function $\phi^{-1}{J}$ satisfies $(\phi^{-1}{J})(\emptyset)=0$ and $(\phi^{-1}{J})(X)=1$ because $J$ is weakly averaging.  

To show  $\phi^{-1}(J)$ is also a countably additive function let $\{U_i\}_{i=1}^{\infty}$ be any disjoint sequence  of  measurable sets in $X$.  Since $J$ is a $\Rinf$-generalized point, $J(s \chi_{U_i}) = s J(\chi_{U_i})$ for any scale factor $s$, and hence we have, using the fact $\sum_{i=1}^{\infty} \frac{1}{2^i} = 1$, that 
 \be \nonumber
J(\chi_U) = J( \displaystyle \sum_{i=1}^{\infty} \chi_{U_i} ) = J( \displaystyle{\sum_{i=1}^{\infty}}  \frac{1}{2^i} (2^i \chi_{U_i}))  
=  \displaystyle{\sum_{i=1}^{\infty}} \frac{1}{2^i} J(2^i \chi_{U_i})  
= \displaystyle{\sum_{i=1}^{\infty}}   J( \chi_{U_i} ) 
 \ee
 where the third line follows from the fact that $J$ is countably additive, and that $\{\frac{1}{2^i}\}_{i=1}^{\infty}$ is a countable partition of one.
Consequently, $\phi^{-1}(J)$ is a probability measure, hence lies in $\G(X)$.
 
 The fact that $\phi$ and $\phi^{-1}$ are $\SCvx$ arrows follows from the pointwise definitions used to define the super convex structure of these two spaces.
\end{proof}

\begin{lemma} The pair of functors $\T$ and $\Sigma$ decompose the Giry monad as \mbox{$\Sigma \circ \T = \G$}.
\end{lemma} 
\begin{proof}
Since $\T$ is $\G$ viewed as a map into $\SCvx$, 
 it is only necessary to verify that the $\sigma$-algebra specified by the functor $\Sigma$ coincides with the $\sigma$-algebra generated by all the evaluation maps  $\{\G(X) \stackrel{ev_U}{\longrightarrow} \Rinf\}_{U \in \Sigma_X}$.
 This follows  from the diagram \ref{diagram} in Lemma \ref{lemmaIso}.   Since $\phi$ is a bijection between the sets, the smallest $\sigma$-algebra on $\G(X)$ such that all the evaluation maps $ev_U$ are measurable corresponds to the smallest $\sigma$-algebra on $\SCvx_{\Rinf^{\Rinf}}(\Rinf^X, \Rinf)$ such that all the evaluation maps $ev_{\chi_U}$ are measurable.   
 The set $\{ev_{\chi_U}\}_{U \in \Sigma_X}$ is a spanning set for the set of all countably convex maps on
  $\SCvx_{\Rinf^{\Rinf}}(\Rinf^X, \Rinf)$.  
  That is,  any countably convex map \mbox{$\SCvx_{\Rinf^{\Rinf}}(\Rinf^X, \Rinf) \stackrel{m}{\longrightarrow} \Rinf$} is an evaluation point, $m=ev_f$, for a unique point $f \in \Rinf^X$ since $\Rinf$ is codense in $\SCvx$.  These functions, like $f$, are completely determined by the characteristic functions. Consequently  the set $\{ev_{\chi_U}\}_{U \in \Sigma_X}$ generates the same $\sigma$-algebra as the set of all countably convex maps.
\end{proof}

\begin{thm} \label{adjunction}
For the functors $\T$ and $\Sigma$, with  the unit specified identically to that of the Giry monad, and the counit, as specified in Lemma \ref{naturalA},  it follows that  $\T \dashv \Sigma$, and $(\T,\Sigma,\eta,\mu)$ is a factorization of the Giry monad.
\end{thm}
\begin{proof}
Verifying the two triangular identities are straightforward.  For $X$ any measurable space we have the commutative $\SCvx$-diagram,
\begin{equation}   \nonumber
 \begin{tikzpicture}[baseline=(current bounding box.center)]
 
 \node        (PX)  at  (0, 0)   {$\T(X)$};
  \node     (PPX) at  (3,0)  {$\T(\Sigma \T(X))$};
   \node   (TX)  at   (3,-1.5)   {$\T(X)$};
   
     \draw[->,above] (PX) to node {\small{$\T\eta_X$}} (PPX);
  \draw[->,below,left] (PX) to node  {\small{$id_{\T X}$}} (TX);
  \draw[->,right] (PPX) to node {\small{$\epsilon_{\T(X)}$}} (TX);
  
  \node        (p)  at  (5, 0)   {$P$};
  \node     (dp) at  (8,0)  {$\delta_P$};
   \node   (p2)  at   (8,-1.5)   {$P$};
   
     \draw[|->] (p) to node {} (dp);
  \draw[|->] (dp) to node  {} (p2);
  \draw[|->] (p) to node {} (p2);
  
 \end{tikzpicture}
 \end{equation}
\noindent
and for $A$ any super convex space, we have the commutative $\Meas$-diagram,
\begin{equation}   \nonumber
 \begin{tikzpicture}[baseline=(current bounding box.center)]
 
 \node        (SPSA)  at  (0, 1.5)   {$\Sigma(\T\Sigma A)$};
  \node     (SA) at  (0,0)  {$\Sigma A$};
   \node   (SA2)  at   (3,1.5)   {$\Sigma A$};
   
     \draw[->,left] (SA) to node {\small{$\epsilon_A$}} (SPSA);
  \draw[->,below,right] (SA) to node  [yshift=-3pt]{\small{$id_{\Sigma A}$}} (SA2);
  \draw[->,above] (SPSA) to node {\small{$\eta_{\Sigma A}$}} (SA2);
  
  \node        (a)  at  (5, 0)   {$a$};
  \node     (da) at  (5,1.5)  {$\delta_a$};
   \node   (a2)  at   (8,1.5)   {$a$};
   
     \draw[|->] (a) to node {} (da);
  \draw[|->] (da) to node  {} (a2);
  \draw[|->] (a) to node {} (a2);
  
 \end{tikzpicture}
 \end{equation}
 Combining this result with the preceding lemma shows that $(\T, \Sigma, \mu, \eta)$ factorizes the Giry monad.
\end{proof}

 \bibliographystyle{plain}

\end{document}